\documentclass{amsart}

\usepackage{eucal}
\usepackage{amssymb}
\usepackage{amsmath}
\usepackage{latexsym}
\usepackage{graphicx}

\newtheorem{theorem}{Theorem}[section]

\newtheorem{cor}[theorem]{Corollary}
\newtheorem{prop}[theorem]{Proposition}

\theoremstyle{definition}
\newtheorem{definition}[theorem]{Definition}
\newtheorem{example}[theorem]{Example}

\theoremstyle{remark}
\newtheorem{remark}[theorem]{Remark}

\numberwithin{equation}{section}

\newcommand{\LL}{{\mathcal L}}

\newcommand{\OO}{{\mathcal O}}
\newcommand{\HC}{{\mathcal H}}
\newcommand{\VV}{{\mathcal V}}
\newcommand{\FC}{{\mathcal F}}

\newcommand{\Z}{\mathbb{Z}}

\newcommand{\Q}{\mathbb{Q}}
\newcommand{\R}{\mathbb{R}}
\newcommand{\C}{\mathbb{C}}

\newcommand{\VB}{\mathbb{V}}
\newcommand{\mh}{\mbox{MHM}}
\newcommand{\po}{\mbox{MH}}

\begin{document}

\title[Hodge genera and characteristic classes]{Hodge-theoretic Atiyah-Meyer formulae \\ and the
stratified multiplicative property}


\author{Laurentiu Maxim}
\address{Institute of Mathematics of the Romanian Academy, PO-BOX 1-764, 70700 Bucharest, Romania.}
\curraddr{Department of Mathematics and Computer Science\\ CUNY-Lehman College\\
250 Bedford Park Blvd West\\ Bronx, New York 10468.}
\email{laurentiu.maxim@lehman.cuny.edu}

\author{J\"org Sch\"urmann}
\address{Mathematische Institut, Universit\"at M\"unster, Einsteinstr. 62, 48149 M\"unster, Germany.}
\email{jschuerm@math.uni-muenster.de}


\dedicatory{Dedicated to L\^e D\~ung Tr\'ang on His $60$th Birthday}


\begin{abstract}
In this note we survey Hodge-theoretic formulae of Atiyah-Meyer type
for genera and characteristic classes of complex algebraic
varieties, and derive some new and interesting applications. We also
present various extensions to the singular setting of the
Chern-Hirzebruch-Serre signature formula.
\end{abstract}

\maketitle

\section{Introduction}
In the mid $1950$'s, Chern, Hirzebruch and Serre \cite{CHS} showed
that if $F \hookrightarrow E \overset{\pi}{\rightarrow} B$ is a
fiber bundle of closed, coherently oriented, topological manifolds
such that the fundamental group of the base $B$ acts trivially on
the cohomology of the fiber $F$, then the signatures of the spaces
involved are related by a simple multiplicative relation:
\begin{equation}\label{CHS}\sigma(E)=\sigma(F) \cdot
\sigma(B).\end{equation}

A decade later, Kodaira \cite{Ko}, Atiyah \cite{At}, and
respectively Hirzebruch \cite{H2} observed that without the
assumption on the (monodromy) action of $\pi_1(B)$ the
multiplicativity relation fails. In the case when $\pi$ is a
differentiable fiber bundle of compact oriented manifolds so that
both $B$ and $F$ are even-dimensional, Atiyah obtained a formula for
$\sigma(E)$ involving a contribution from the monodromy action. 
Let $k=\frac{1}{2}{\rm dim_{\R}}F$. Then the flat bundle
$\VV$ over $B$ with fibers $H^k(F_x;\R)$ $(x \in B)$ has a
$K$-theory signature, $[\VV]_K \in KO(B)$ for $k$ even (resp. in
$KU(B)$ for $k$ odd), and the Atiyah signature theorem \cite{At}
asserts that
\begin{equation}\label{A} \sigma(E)=\langle ch^*_{(2)}([\VV]_K) \cup L^*(B), [B]
\rangle, \end{equation} where $ch^*_{(2)}$ is a modified Chern
character (obtained by precomposing with the second Adams operation),
and $L^*(B)$ is the total Hirzebruch $L$-polynomial of $B$.\\

Meyer \cite{Me} extended Atiyah's formula to the case of
\emph{twisted signatures} of closed manifolds endowed with
Poincar\'e local systems (that is, local systems with duality) not
necessarily arising from a fibre bundle projection. If $B$ is a
closed, oriented, smooth manifold of even dimension, and $\LL$ is a
local system equipped with a nondegenerate (anti-)symmetric bilinear
pairing $\LL \otimes \LL \to \R_B$, then the twisted signature
$\sigma(B;\LL)$ is defined to be the signature of the nondegenerate
form on the sheaf cohomology group $H^{dim(B)/2}(B;\LL)$, and can be
computed by Meyer's signature formula:
\begin{equation}\label{M} \sigma(B;\LL)=\langle ch^*_{(2)}([\LL]_K) \cup L^*(B), [B]
\rangle, \end{equation} where $[\LL]_K$ is the $K$-theory signature
of $\LL$ defined as follows. For $k$ even (resp. for $k$ odd) the nondegenerate
pairing induces a splitting of the associated flat bundle 
$\LL= \LL^+  \oplus \LL^- $ into a positive and negative definite part
(resp. induces a complex structure on the associated flat bundle 
$\LL$ with $\LL^*$ the complex conjugate bundle). Then
$$ [\LL]_K:= \begin{cases} \LL^+  - \LL^- \in  KO(B) & \text{if $k$ is even,}\\
\LL^*  - \LL \in  KU(B) & \text{if $k$ is odd.}\end{cases}$$

Geometric mapping situations that involve singular spaces generally
lead to Poincar\'e local systems that are only defined on the top
stratum of a stratified space. For example, Cappell and Shaneson
\cite{CS2} proved that if $f:Y \to X$ is a stratified map of even
relative dimension between oriented, compact, Whitney stratified
spaces with only strata of even codimension, then:
\begin{equation}\label{CS}
\sigma(Y)=\sigma(X;{\LL}^f_{X - \Sigma}) + \sum_{\text{pure strata }
Z \subset X} \sigma (\bar{Z}; \LL^f_Z),\end{equation} where $\Sigma
\subset X$ is the singular set of $f$ and, for an open stratum $Z$
in $X$, $\LL^f_Z$ is a certain Poincar\'e local system defined on
it. In particular, if all strata $Z \subset X$ of $f$ are
simply-connected then, as an extension of the Chern-Hirzebruch-Serre
formula (\ref{CHS}) to the stratified case, we obtain from
(\ref{CS}) that
\begin{equation}\label{CSs}
\sigma(Y)=\sum_{\text{ strata } Z \subset X} \sigma (\bar{Z}) \cdot
\sigma(N_Z),
\end{equation}
where for a pure stratum $Z$ of real codimension at least two and
with link $L_Z$ in $X$,
$$N_Z:=f^{-1}({\rm cone } \ L_{Z}) \cup_{f^{-1}(L_{Z})} {\rm
cone } \ f^{-1}(L_{Z}) $$ is the topological completion of the
preimage under $f$ of the normal slice to $Z$ in $X$; if $Z$ is a
component of the top stratum $X \setminus \Sigma$, then $N_Z$ is the
fiber of $f$ over $Z$.\\

 More generally, similar formulae hold for the
push-forward of the Goresky-MacPherson $L$-classes $L_k(Y) \in
H_k(Y;\mathbb{Q})$, $k \geq 0$. On a space $X$ with only
even-codimensional strata and singular set $\Sigma$, the twisted
homology $L$-classes $L_k(X;\LL)$ and the twisted signature
$\sigma(X;\LL)$ for a Poincar\'e local system $\LL$ on $X-\Sigma$
can be defined by noting that the duality of the local system
extends to a self-duality of the corresponding middle-perversity
intersection chain sheaf $IC_X(\LL)$ on $X$ (for complete details on
the construction, the reader is advised to consult the book
\cite{Ba} and the references therein).\\

It is therefore natural at this point to ask for extensions of
Meyer's signature formula to the singular setting. In \cite{BCS},
Banagl, Cappell and Shaneson proved the following.
Suppose $X$ is a closed oriented Whitney stratified normal Witt
space (that is, a space on which the middle-perversity intersection
chain sheaf $IC_X$ is self-dual, cf. \cite{Si}) of even dimension
with singular set $\Sigma$, and let $\LL$ be a Poincar\'e local
system defined on $X-\Sigma$ such that $\LL$ is \emph{strongly
transverse to $\Sigma$}. On normal spaces, this technical assumption
is equivalent to saying that $\LL$ has a unique extension as a
Poincar\'e local system to all of $X$. Such a local system possesses
a $K$-theory signature $[\LL]_K$ in the $K$-theory of $X$ (cf.
\cite{BCS}, Corollary $2$), and $IC_X(\LL)$ is again self-dual. Then
the twisted $L$-classes are well-defined, and they can be computed by the
formula (cf. \cite{BCS}, Theorems $1$ and $3$)
\begin{equation}\label{BCS}
L_*(X;\LL)=ch^*_{(2)}\left([\LL]_K\right) \cap L_*(X)
\end{equation} (here $L_*$ stands for the total homology $L$-classes respectively);
in particular, the twisted signature is given by
\begin{equation}\label{BCS2}
\sigma(X;\LL)=\langle ch^*_{(2)}\left([\LL]_K\right), L_*(X)
\rangle.
\end{equation}

\bigskip

In this note, we survey Hodge theoretic Atiyah-Meyer type formulae
for genera and characteristic classes of complex algebraic
varieties. In fact, these are Hodge theoretic analogues of the above
formulae (see \cite{CLMS,CLMS2}), and various extensions to the
singular setting (see \cite{MS}). We also present the main ideas and
constructions that lead to the stratified multiplicative property
for Hodge genera and the Hirzebruch characteristic classes of
complex algebraic varieties; for more details on part of this work,
see \cite{CMS0,CMS}. Some of the results in this note were announced
in the present form in the paper \cite{CLMS2}.

The first author is grateful to
his mentors and collaborators Sylvain Cappell, Anatoly Libgober and Julius
Shaneson for their contribution to the work summarized in this
report, and for constant guidance and support.

\section{Hirzebruch characteristic classes.}
In this section we first define the Hirzebruch class of a smooth
complex projective algebraic variety, then, following \cite{BSY,SY},
we describe its recent generalization to the singular setting. The
construction in the singular case yields characteristic classes in
(Borel-Moore) homology, and makes use of Saito's theory of algebraic
mixed Hodge modules. In this section, we only survey formal
properties of this deep theory which will be needed in the sequel.

\subsection{The non-singular case.}\label{Hir} If $Z$ is a smooth
projective complex algebraic variety, the signature and the
$L$-classes of $Z$ are special cases of more general Hodge theoretic
invariants encoded by the Hirzebruch characteristic class (also
called ``the generalized Todd class") $T_y^*(T_Z)$ of the tangent
bundle of $Z$ (cf. \cite{H}). This is defined by the normalized
power series
\begin{equation}
Q_y(\alpha)=\frac{\alpha(1+y)}{1-e^{-\alpha(1+y)}}-\alpha y \in
\Q[y][[\alpha]],
\end{equation}
that is, \begin{equation}\label{Hsm}
T_y^*(T_Z)=\prod_{i=1}^{dim(Z)}Q_y(\alpha_i),
\end{equation}
where $\{\alpha_i\}$ are the Chern roots of the tangent bundle
$T_Z$. Note that $Q_{y}(\alpha)$ is equal to $1+\alpha$ for $y=-1$,
to $\frac{\alpha}{1-e^{-\alpha}}$ for $y=0$, and it equals
$\frac{\alpha}{\tanh \alpha}$ if $y=1$. Therefore, the Hirzebruch
class $T_y^*(T_Z)$ coincides with the total Chern class $c^*(T_Z)$
if $y=-1$, with the total Todd class $td^*(T_Z)$ if $y=0$, and 
with the total Thom-Hirzebruch $L$-class $L^*(T_Z)$ if $y=1$.\\

The Hirzebruch class appears in the \emph{generalized
Hirzebruch-Riemann-Roch theorem} (cf. \cite{H}, \S 21.3), which
asserts that if $\Xi$ is a holomorphic vector bundle on a smooth
complex projective variety $Z$, then the $\chi_y$-characteristic of
$\Xi$, which is defined by
\begin{equation}\label{Hichi}\chi_y(Z,\Xi):=\sum_{p \geq 0} \chi(Z,
\Xi \otimes \Lambda^pT^*_Z) \cdot y^p=\sum_{p \geq 0} \left( \sum_{i
\geq 0} (-1)^i \text{dim} H^i(Z,\Omega(\Xi) \otimes \Lambda^pT^*_Z)
\right) \cdot y^p,\end{equation} with $T^*_Z$ the holomorphic
cotangent bundle of $Z$ and $\Omega(\Xi)$ the coherent sheaf of
germs of sections of $\Xi$, can in fact be expressed in terms of the
Chern classes of $\Xi$ and the tangent bundle of $Z$, or more
precisely,
\begin{equation}\label{gHRR}
\chi_y(Z,\Xi)=\langle ch_{(1+y)}^*(\Xi) \cup
{T}_y^*(T_Z),[Z]\rangle,
\end{equation}
where $ch^*_{(1+y)}(\Xi)=\sum_{j=1}^{{\mbox{rk}} (\Xi)} e^{\beta_j (1+y)},$ for
$\{\beta_j\}_j$ the Chern roots of $\Xi$. In particular, if
$\Xi=\mathcal{O}_Z$, the Hirzebruch genus
$\chi_y(Z):=\chi_y(Z,\mathcal{O}_Z)$ can be computed by
\begin{equation}\label{Hchi}\chi_y(Z)=\langle {T}_y^*(T_Z) , [Z] \rangle.\end{equation}

\subsection{Mixed Hodge modules.} Before discussing extensions of
the Hirzebruch class to the singular setting, we need to briefly
recall some aspects of Saito's theory of algebraic mixed Hodge
modules. Generic references for this theory are Saito's papers
\cite{Sa0,Sa2,Sa1}.\\

To each complex algebraic variety $Z$, M. Saito associated an abelian
category $\mh(Z)$ of algebraic mixed Hodge modules on $Z$ (cf.
\cite{Sa0,Sa1}). If $Z$ is smooth, an object of this category
consists of a bifiltered regular holonomic $D$-module $(M,W,F)$
together with a filtered perverse sheaf $(K,W)$ that corresponds,
after tensoring with $\C$, to $(M,W)$ under the Riemann-Hilbert
correspondence. In general, for a singular variety $Z$ one works
with suitable local embeddings into manifolds and corresponding
filtered $D$-modules supported on $Z$. In addition, these objects
are required to satisfy a long list of complicated properties.\\

The forgetful functor from $\mh(Z)$ to the category of perverse
sheaves extends to a functor $rat:D^b\mh(Z) \to D^b_c(Z)$ to the
derived category of complexes of $\Q$-sheaves with constructible
cohomology. The usual truncation $\tau_{\leq}$ on $D^b\mh(Z)$
corresponds to the perverse truncation ${^p\tau}_{\leq}$ on
$D^b_c(Z)$. Saito also constructed a $t$-structure $\tau'_{\leq}$ on
$D^b\mh(Z)$ which is compatible with the usual $t$-structure on
$D^b_c(Z)$ (\cite{Sa1}, Remark 4.6(2)). There are functors $f_*$,
$f_!$, $f^*$, $f^!$, $\otimes$, $\boxtimes$ on $D^b\mh(Z)$ which are
``lifts" via $rat$ of the similar functors defined on $D^b_c(Z)$. If
$f$ is a proper algebraic morphism then $f_*=f_!$.\\

It follows from the definition that every $M \in \mh(Z)$ has an
increasing weight filtration $W$ so that the functor $M \to Gr^W_kM$
is exact. We say that $M \in \mh(Z)$ is pure of weight $k$ if
$Gr_i^WM=0$ for all $i \neq k$. The weight filtration is extended to
the derived category $D^b\mh(Z)$ by requiring that a shift $M
\mapsto M[1]$ increases the weights by one. So $M \in D^b\mh(Z)$ is
pure of weight $k$ if $H^i(M)$ is pure of weight $i+k$ for all $i
\in \Z$. If $f$ is a map of algebraic varieties, then $f_!$ and
$f^*$ preserve weight $\leq k$, and $f_*$ and $f^!$ preserve weight
$\geq k$. In particular, if $M \in D^b\mh(X)$ is pure and $f:X \to
Y$ is proper, then $f_*M \in D^b\mh(Y)$ is pure of the same weight
as $M$.\\

We say that $M \in \mh(Z)$ is supported on $S$ if and only if
$rat(M)$ is supported on $S$.
There are abelian subcategories $\po(Z,k)^p \subset \mh(Z)$ of pure
polarizable Hodge modules of weight $k$. For each $k \in \Z$, the
abelian category $\po(Z,k)^p$ is semi-simple, in the sense that every
polarizable Hodge module on $Z$ can be uniquely written as a direct
sum of polarizable Hodge modules with strict support in irreducible
closed subvarieties of $Z$. Let $\po_S(Z,k)^p$ denote the subcategory
of polarizable Hodge modules of weight $k$ with strict support in
$S$. Then every $M \in \po_S(Z,k)^p$ is generically a polarizable
variation of Hodge structures $\VB_U$ on a Zariski dense open subset
$U \subset S$, with quasi-unipotent monodromy at infinity.
Conversely, every such polarizable variation of Hodge structures can
be extended in an unique way to a pure Hodge module. Under this
correspondence, for $M \in \po_S(Z,k)^p$ we have that
$rat(M)=IC_S(\VB)$, for $\VB$ the corresponding variation of Hodge
structures.\\

Saito showed that the category of mixed Hodge modules supported on a
point, $\mh(pt)$, coincides with the category $mHs^p$ of (graded)
polarizable rational mixed Hodge structures. Here one has to switch
the increasing $D$-module filtration $F_*$ of the mixed Hodge module
to the decreasing Hodge filtration of the mixed Hodge structure by
$F^{*}:=F_{-*}$, so that $gr^p_F \simeq  gr^F_{-p}$. In this case,
the functor $rat$ associates to a mixed Hodge structure the
underlying rational vector space. There exists a unique object $\Q^H
\in \mh(pt)$ such that $rat(\Q^H)=\Q$ and $\Q^H$ is of type $(0,0)$.
In fact, $\Q^H=((\C,F), \Q, W)$, with $gr^F_i=0=gr^W_i$ for all $i
\neq 0$. For a complex variety $Z$, define $\Q_Z^H:=a_Z^*\Q^H \in
D^b\mh(Z)$, with $rat(\Q_Z^H)=\Q_Z$, for $a_Z:Z\to pt$ the map to a
point. If $Z$ is \emph{smooth} of complex dimension $n$ then
$\Q_Z[n]$ is perverse on $Z$, and $\Q_Z^H[n]\in \mh(Z)$ is a single
mixed Hodge module (in degree $0$), explicitly described by
$\Q_Z^H[n]=((\OO_Z, F), \Q_Z[n], W),$ where $F$ and $W$ are trivial
filtrations so that $gr^F_i=0=gr^W_{i+n}$ for all $ i \neq 0$. So if
$Z$ is smooth of dimension $n$, then $\Q_Z^H[n]$ is a pure mixed
Hodge module of weight $n$. Next, note that if $j: U \hookrightarrow
Z$ is a Zariski-open dense subset in $Z$, then the intermediate
extension $j_{!*}$ (cf. \cite{BBD}) preserves the weights. This
shows that if $Z$ is a complex algebraic variety of pure dimension
$n$ and $j: U \hookrightarrow Z$ is the inclusion of a smooth
Zariski-open dense subset then the intersection cohomology module
$IC_Z^H:=j_{!*}(\Q_U^H[n])$ is pure of weight $n$, with underlying
perverse sheaf $rat(IC_Z^H)=IC_Z$.\\

If $Z$ is smooth of dimension $n$, an object $M \in \mh(Z)$ is
called \emph{smooth} if and only if $rat(M)[-n]$ is a local system
on $Z$. Smooth mixed Hodge modules are (up to a shift) admissible
(at infinity) variations of mixed Hodge structures (in the sense of
Steenbrink-Zucker \cite{SZ} and Kashiwara \cite{K}). Conversely, an
admissible variation of mixed Hodge structures $\LL$ (e.g., a
geometric variation, or a pure polarizable variation) on a smooth
variety $Z$ of pure dimension $n$ gives rise to a smooth mixed Hodge
module (cf. \cite{Sa1}), i.e., to an element $\LL^H[n]\in \mh(Z)$
with $rat(\LL^H[n])=\LL[n]$. A pure polarizable variation of weight
$k$ yields an element of $\po(Z,k+n)^p$. By the stability by the
intermediate extension functor it follows that if $Z$ is an
algebraic variety of pure dimension $n$ and $\LL$ is an admissible
variation of (pure) Hodge structures (of weight $k$) on a smooth
Zariski-open dense subset $U \subset Z$, then $IC^H_Z(\LL)$ is an
algebraic mixed Hodge module (pure of weight $k+n$), so that
$rat(IC_Z^H(\LL)|_U)=\LL[n]$.

\subsection{Grothendieck
groups of algebraic mixed Hodge modules.}\label{Grot} In this
section, we describe the functorial calculus of Grothendieck groups
of algebraic mixed Hodge modules. Let $Z$ be a complex algebraic
variety. By associating to (the class of) a complex the alternating
sum of (the classes of) its cohomology objects, we obtain the
following identification (e.g. compare [\cite{KS}, p. 77],
[\cite{Sc}, Lemma 3.3.1])
\begin{equation} K_0(D^b\mh(Z))=K_0(\mh(Z)).
\end{equation}
In particular, if $Z$ is a point, then
\begin{equation} K_0(D^b\mh(pt))=K_0(mHs^p),
\end{equation}
and the latter is a commutative ring with respect to the tensor
product, with unit $[\Q^H_{pt}]$. Let $\tau_{\leq}$ be the natural
truncation on $D^b\mh(Z)$ with associated cohomology $H^{*}$.  Then
for any complex $M^{\bullet} \in D^b\mh(Z)$ we have the
identification
\begin{equation}\label{i1}
[M^{\bullet}]=\sum_{i \in \Z} (-1)^i [H^i(M^{\bullet})] \in
K_0(D^b\mh(Z)) \cong K_0(\mh(Z)).
\end{equation}
In particular, if for any $M \in \mh(Z)$ and $k \in \Z$ we regard
$M[-k]$ as a complex concentrated in degree $k$, then
\begin{equation}\label{i2}
\left[ M[-k] \right]= (-1)^k [M] \in K_0(\mh(Z)).
\end{equation}
All functors $f_*$, $f_!$, $f^*$, $f^!$, $\otimes$, $\boxtimes$
induce corresponding functors on $K_0(\mh(\cdot))$. Moreover,
$K_0(\mh(Z))$ becomes a $K_0(\mh(pt))$-module, with the
multiplication induced by the exact exterior product
$$\boxtimes : \mh(Z) \times \mh(pt) \to \mh(Z \times \{pt\}) \simeq
\mh(Z).$$ Also note that $$M \otimes \Q^H_Z \simeq M \boxtimes
\Q^H_{pt} \simeq M$$ for all $M \in \mh(Z)$. Therefore,
$K_0(\mh(Z))$ is a unitary $K_0(\mh(pt))$-module. The functors
$f_*$, $f_!$, $f^*$, $f^!$ commute with exterior products (and $f^*$
also commutes with the tensor product $\otimes$), so that the
induced maps at the level of Grothendieck groups $K_0(\mh(\cdot))$
are $K_0(\mh(pt))$-linear. Moreover, by the functor
$$rat:K_0(\mh(Z)) \to K_0(D^b_c(Z)) \simeq K_0(Perv(\Q_Z)),$$ these  transformations lift the corresponding ones from the (topological) level
of Grothendieck groups of constructible (or perverse) sheaves.

\subsection{Hirzebruch classes in the singular
setting.}\label{singular} For any complex variety $Z$, and for any
$p \in \Z$, Saito constructed a functor of triangulated categories
\begin{equation} gr^F_pDR: D^b\mh(Z) \to D^b_{coh}(Z)\end{equation}
commuting with proper push-down,
with $gr^F_pDR(M)=0$ for almost all $p$ and $M$ fixed,
where $D^b_{coh}(Z)$ is the bounded
derived category of sheaves of $\mathcal{O}_Z$-modules with coherent
cohomology sheaves.
If $\Q_Z^H \in D^b\mh(Z)$ denotes the constant
Hodge module on $Z$, and if $Z$ is smooth and pure dimensional, then
$gr^F_{-p} DR(\Q_Z^H) \simeq \Omega^p_Z[-p]$. The transformations
$gr^F_pDR$ induce functors on the level of Grothendieck groups.
Therefore, if $G_0(Z) \simeq K_0(D^b_{coh}(Z))$ denotes the
Grothendieck group of coherent sheaves on $Z$, we get a group
homomorphism (the {\em motivic Chern class transformation})
\begin{equation}\label{grF}
MHC_*: K_0(\mh(Z)) \to G_0(Z) \otimes \Z[y, y^{-1}]\;;
\end{equation}
$$[M] \mapsto \sum_{i,p} (-1)^{i} [\HC^i ( gr^F_{-p} DR(M) )] \cdot
(-y)^p\:.$$
We let $td_{(1+y)}$ be the natural transformation (cf. \cite{Y,BSY})
\begin{equation}\label{td}
td_{(1+y)}:G_0(Z) \otimes \Z[y, y^{-1}] \to H^{BM}_{2*}(Z) \otimes
\Q[y, y^{-1}, (1+y)^{-1}]\;;
\end{equation}
$$[\mathcal{F}] \mapsto \sum_{k \geq
0} td_k([\mathcal{F}]) \cdot (1+y)^{-k}\:,$$
where $H_*^{BM}$ stands for Borel-Moore homology, and $td_k$ is the
degree $k$ component (i.e., in $H_{2k}^{BM}(Z)$) of the Todd class
transformation $td_*:G_0(Z) \to H_{2*}^{BM}(Z) \otimes \Q$ of
Baum-Fulton-MacPherson \cite{BFM}, which is linearly extended over
$\Z[y, y^{-1}]$.

\begin{definition}\label{d1} The motivic Hirzebruch class
transformation $MHT_y$ is defined by the composition (cf.
\cite{BSY})
\begin{equation}\label{IT}
MHT_y :=td_{(1+y)} \circ MHC_*: K_0(\mh(Z)) \to H^{BM}_{2*}(Z)
\otimes \Q[y,y^{-1},(1+y)^{-1}]\:.
\end{equation}
The motivic Hirzebruch class ${T_y}_*(Z)$ of a complex algebraic
variety $Z$ is defined by \begin{equation}
{T_y}_*(Z):=MHT_y([\Q_Z^H]).
\end{equation}
Similarly, if $Z$ is an $n$-dimensional complex algebraic manifold,
and $\LL$ is a local system on $Z$ underlying an admissible
variation of mixed Hodge structures, we define twisted Hirzebruch
characteristic classes by
\begin{equation}\label{tHc} {T_y}_*(Z; \LL)=MHT_y([\LL^H]),\end{equation} where
$\LL^H[n]$ is the smooth mixed Hodge module on $Z$ with underlying
perverse sheaf $\LL[n]$.
\end{definition}

\begin{example}\label{pt} Let $\VB=((V_{\C},F), V_{\Q},K) \in \mh(pt)=mHs^p$. Then:
\begin{equation}\label{point}
MHT_y([\VB])=\sum_p td_0([gr^p_F V_{\C}]) \cdot (-y)^p = \sum_p
\text{dim}_{\C} (gr^p_F V_{\C}) \cdot (-y)^p = \chi_y([\VB]),
\end{equation}
so over a point the transformation $MHT_y$ coincides with the
$\chi_y$-genus ring homomorphism $\chi_y:K_0(mHs^p) \to
\Z[y,y^{-1}]$.
\end{example}

By definition, the transformations $MHC_*$ and $MHT_y$   commute with proper
push-forward. The following {\it normalization} property holds (cf.
\cite{BSY}): If $Z$ is smooth and pure dimensional, then
$${T_y}_*(Z)=T_y^*(T_Z) \cap [Z]\:,$$
where $T_y^*(T_Z)$ is the
cohomology Hirzebruch class of $Z$ defined in \S\ref{Hir}. So, if
$Z$ is smooth and projective, then $T_1^*(T_Z)$ is the total
Hirzebruch $L$-polynomial of $Z$ and $\chi_1(Z)=\sigma(Z)$.

For a complete (possibly singular) variety $Z$ with $k:Z \to pt$ the
constant map to a point, the pushdown $k_*{T_y}_*(Z)$ is the Hodge
genus
\begin{equation}\label{defchi}\chi_y(Z):=\chi_y([H^*(Z;\Q)])=\sum_{i,p}
(-1)^i dim_{\C} (gr^p_F H^i(Z;\C)) \cdot (-y)^p,
\end{equation}
with $\chi_{-1}(Z):=\chi([H^*(Z;\Q)])$ the topological Euler characteristic of $Z$.
For $Z$ smooth $k_*{T_y}_*(Z;\LL)$ is the twisted $\chi_y$-genus
$\chi_y(Z;\LL)$ defined in a similar manner (\cite{CLMS})
\footnote{Note that by Deligne's theory, if $Z$ is smooth and
projective then $\chi_y(Z)$ defined in (\ref{defchi}) yields the
same invariant as $\chi_y(Z;\mathcal{O}_Z)$ defined by the equation
(\ref{Hichi}).}.

It was shown in \cite{BSY} that for any variety $Z$ the limits $T_{y*}(Z)$
for $y=-1,0$ exist, with
$${T_{-1}}_*(Z)=c_*(Z) \otimes \Q$$
the total (rational) Chern
class of MacPherson (for a construction of the latter see \cite{M}).
Moreover, for a variety $Z$ with at most Du Bois singularities
(e.g., toric varieties), we have that
$${T_0}_*(Z)=td_*(Z):=td_*([\mathcal{O}_Z]) \:,$$
for  $td_*$ the
Baum-Fulton-MacPherson transformation \cite{BFM}. It is still a
conjecture that for a rational homology manifold ${T_1}_*(Z)$
coincides with the total Goresky-MacPherson homology $L$-class of
$Z$ (see \cite{BSY}, p.4 and Remark 5.4). As will be shown elsewhere, this conjecture is true at least for
$Z=M/G$ the quotient of a complex projective manifold $M$ by the algebraic action
of a finite group.\\

The Hirzebruch class of Section \ref{Hir} also admits another
extension to the singular setting, which is defined by means of
intersection homology. Let $IC^H_Z \in \mh(Z)$ be the intersection
homology (pure) Hodge module on a pure-dimensional variety $Z$, so
$rat(IC^H_Z)=IC_Z$. Similarly, for an admissible variation of mixed
Hodge structures $\LL$ defined on a smooth Zariski dense open subset
of $Z$, let $IC_Z^H(\LL)$ be the corresponding mixed Hodge module
with underlying perverse sheaf $IC_Z(\LL)$. In order to simplify the
notations in the following definition, we set
$$IC'^H_Z:=IC^H_Z[-{\rm dim}_{\C} Z] \  \ \ \text{and} \ \ \
IC'^H_Z(\LL):=IC^H_Z(\LL)[-{\rm dim}_{\C} Z].$$
\begin{definition}\label{IHinv} We define intersection characteristic classes by
\begin{equation}{IT_y}_*(Z):=MHT_y(\left[ IC'^H_Z
\right]) \in H^{BM}_{2*}(Z) \otimes
\Q[y,y^{-1},(1+y)^{-1}],\end{equation} and similarly,
\begin{equation}{IT_y}_*(Z;\LL):=MHT_y(\left[ IC'^H_Z(\LL)
\right]),\end{equation} for $\LL$ an admissible variation of mixed
Hodge structures defined on a smooth Zariski dense open subset of
$Z$.
\end{definition}
As we will see later on, the limit ${IT_y}_*(Z;\LL)$ for $y=-1$ always exists
(as well as ${IT_y}_*(Z;\LL)$ for $y=0$, if $\LL$ is of non-negative weight,
e.g. $\LL=\Q_Z$).
If $Z$ is complete, then by pushing ${IT_y}_*(Z)$ down to a point we
recover the intersection homology $\chi_y$-genus, $I\chi_y(Z)$,
which is a polynomial in the Hodge numbers of $IH^*(Z;\Q)$ defined
by $$I\chi_y(Z):=\chi_y([IH^*(Z;\Q)])$$ Similarly, in the above
notations and if $Z$ is complete, one has that
$$I\chi_y(Z;\LL)=k_*{IT_y}_*(Z;\LL)\:,$$
for $k:Z \to pt$ the constant map.
Note that $I\chi_{-1}(Z)=\chi([IH^*(Z;\Q)]$ for $Z$ complete is the
intersection (co)homology Euler characteristic of $Z$, whereas for $Z$
projective, $I\chi_1(Z)$ is the intersection (co)homology
signature of $Z$ due to Goresky-MacPherson. If $Z$ is a $\Q$-homology
manifold, then
$$\Q^H_Z \simeq IC'^H_Z \in D^b\mh(Z) \:,$$
so we get that
${T_y}_*(Z)={IT_y}_*(Z)$. It is conjectured that for a compact
variety $Z$, ${IT_1}_*(Z)$ is the Goresky-MacPherson homology
$L$-class $L_*(Z)$ (\cite{BSY}, Remark 5.4).

\section{The stratified multiplicative property.}\label{SMP}
In this section we give a brief survey of the main ideas and results
concerning the behavior of the singular Hirzebruch classes under
proper algebraic morphisms. The main references are the papers
\cite{CMS0,CMS}. Similar results were originally predicted by
Cappell and Shaneson (cf. \cite{CS,S}), and were referred to as
``the stratified multiplicative property for $\chi_y$-genera and
Hirzebruch characteristic classes". The results surveyed in this
section are motivated by the attempt of adapting the
Cappell-Shaneson formulae (\ref{CS}) and (\ref{CSs}) for the
(topological) signature and $L$-classes to the setting of complex
algebraic (analytic) geometry.\\

Let $Y$ be an irreducible complex algebraic variety endowed with a
complex algebraic Whitney stratification $\VV$ so that the
intersection cohomology complexes $$IC'_{\bar W}:=IC_{\bar W}[-{\rm
dim}_{\C}(W)]$$ are $\VV$-constructible for all strata $W \in \VV$.
(All these complexes are regarded as complexes on all of $Y$.)
Define a partial order on $\VV$ by ``$V \leq W$ if and only if $V
\subset \bar W$". Denote by $S$ the top-dimensional stratum, so $S$
is Zariski open and dense, and $V \leq S$ for all $V \in \VV$. Let
us fix for each $W \in \VV$ a point $w \in W$ with inclusion
$i_w:\{w\} \hookrightarrow Y$. Then
\begin{equation}\label{100} i_w^*[IC'^H_{\bar W}]=[i_w^*IC'^H_{\bar
W}]=[\Q^H_{pt}]\in K_0(\mh(w))=K_0(\mh(pt)),\end{equation} and
$i_w^*[IC'^H_{\bar V}] \neq [0] \in K_0(\mh(pt))$ only if $W \leq
V$. Moreover,  for any $j \in \Z$, we have
\begin{equation}\label{cone} \HC^j (i_w^*IC'_{\bar V}) \simeq
IH^j(c^{\circ}L_{W,V}),\end{equation} with $c^{\circ}L_{W,V}$ the
open cone on the link $L_{W,V}$ of $W$ in $\bar V$ for $W \leq V$
(cf. \cite{B}, p.30, Prop. 4.2). So
$$i_w^*[IC'^H_{\bar V}]=[IH^*(c^{\circ}L_{W,V})] \in K_0(\mh(pt)),$$
with the mixed Hodge structures on the right hand side defined by
the isomorphism (\ref{cone}).

The main technical result of this section is the following
\begin{theorem}(\cite{CMS}, Thm. 3.2) \label{main}  For each stratum $V \in \VV \setminus
\{S\}$ define inductively
\begin{equation}\label{eq8}
\widehat{IC^H}(\bar V):=[IC'^H_{\bar V}] - \sum_{W < V}
\widehat{IC^H}(\bar W) \cdot i_w^* [IC'^H_{\bar V}] \in
K_0(D^b\mh(Y)).
\end{equation}
Assume $[M] \in K_0(D^b\mh(Y))$ is an element of the
$K_0(\mh(pt))$-submodule $\langle [IC'^H_{\bar V}] \rangle$ of
$K_0(D^b\mh(Y))$ generated by the elements $[IC'^H_{\bar V}]$, $V
\in \VV$. Then we have the following equality  in $K_0(D^b\mh(Y))$:
\begin{equation}\label{mE}
[M]= [IC'^H_Y] \cdot i_s^*[M]+\sum_{V < S}  \widehat{IC^H}(\bar V)
\cdot \left( i_v^*[M] -i_s^*[M] \cdot i_v^*[IC'^H_Y] \right).
\end{equation}
\end{theorem}

Before stating immediate consequences of the above theorem, let us
recall from \cite{CMS} some cases when the technical hypothesis $[M]
\in \langle [IC'^H_{\bar V}] \rangle$ is satisfied for a fixed $M
\in D^b(\mh(Y))$. Assume that all sheaf complexes $IC'_{\bar V}$, $V
\in \VV$, are not only $\VV$-constructible, but satisfy the stronger
property that they are ``cohomologically $\VV$-constant", i.e., all
cohomology sheaves $\HC^j(IC'_{\bar V})|_W$ ($j \in \Z$) are
constant for all $V,W \in \VV$ (e.g., this is the case if $Y$ is a
toric variety with its natural Whitney stratification by orbits, cf.
\cite{BL}). Moreover, assume that either
\begin{enumerate}\item $rat(M)$ is also cohomologically
$\VV$-constant, or
\item all perverse cohomology sheaves $rat(H^j(M))$ ($j \in \Z$) are
cohomologically $\VV$-constant, e.g., each $H^j(M)$ is a pure Hodge
module with the property that $\HC^{-dim(V)}(rat(H^j(M))|_V)$ is constant for all $V \in
\VV$.\end{enumerate} Then $[M] \in \langle [IC'^H_{\bar V}]
\rangle$. In particular, if all strata $V \in \VV$ are
simply-connected, then we have that $[M] \in \langle [IC'^H_{\bar
V}] \rangle$ for all $M \in D^b\mh(X)$ so that $rat(M)$ is
$\VV$-constructible.

\bigskip

In the following, we specialize to the relative context of a proper
algebraic map $f:X \to Y$ of complex algebraic varieties, with $Y$
irreducible. For a given $M \in D^b\mh(X)$, assume that $Rf_*rat(M)$
is constructible with respect to the given complex algebraic Whitney
stratification $\VV$ of $Y$, with open dense stratum $S$. By proper
base change, we get
$$i_v^*f_*[M]=[H^*(\{f=v\}, rat(M))] \in K_0(\mh(pt)).$$
So under the assumption $f_*[M] \in \langle [IC'^H_{\bar V}]
\rangle$, Theorem \ref{main} yields the following identity in
$K_0(\mh(Y))$:
\begin{cor}\label{cormain}
\begin{eqnarray*}
f_*[M]&=& [IC'^H_Y] \cdot [H^*(F;rat(M))]  \\ &+& \sum_{V < S}
\widehat{IC^H}(\bar V) \cdot \left( [H^*(F_V;rat(M))]
-[H^*(F;rat(M))] \cdot [IH^*(c^{\circ}L_{V,Y})] \right),
\end{eqnarray*}
where $F$ is the (generic) fiber over the top-dimensional stratum
$S$, and $F_V$ is the fiber over a stratum $V \in \VV \setminus
\{S\}$.
\end{cor} Note that the corresponding classes $[H^*(F;rat(M))]$ and
$[H^*(F_V;rat(M))]$ may depend on the choice of fibers of $f$, but
the above formula holds for any such choice. If all strata $V \in
\VV$ are simply connected, then these classes are independent of the
choices made. By pushing the identity in Corollary \ref{cormain}
down to a point via $k'_*$, for $k':Y \to pt$ the constant map, and
using the fact that $k'_*$ is $K_0(\mh(pt))$-linear, an application
of the $\chi_y$-genus (ring) homomorphism yields the following:
\begin{prop}\label{gsmp} Under the above notations and assumptions, the following
identity holds in $\Z[y,y^{-1}]$:
\begin{eqnarray*}
&& \chi_y([H^*(X;rat(M)])= I\chi(Y) \cdot \chi_y([H^*(F;rat(M))]) \\
&+& \sum_{V < S} \widehat{I\chi}_y (\bar V) \cdot \left(
\chi_y([H^*(F_V;rat(M))]) -\chi_y([H^*(F;rat(M))]) \cdot
I\chi_y(c^{\circ}L_{V,Y}) \right),
\end{eqnarray*}
where for $V < S$, $\widehat{I\chi}_y (\bar V)$ is defined
inductively by $$\widehat{I\chi}_y(\bar V)= I\chi_y(\bar V)- \sum_{W
< V} \widehat{I\chi}_y(\bar W) \cdot I\chi_y(c^{\circ}L_{W,V}).$$
\end{prop}

In particular, if in Proposition \ref{gsmp} we take $M=\Q^H_X$, we
obtain the following \footnote{Here we use the deep result due to
Saito \cite{Sa5} that Deligne's and Saito's mixed Hodge structures
on cohomology groups coincide.}
\begin{theorem}(\cite{CMS}, Thm. 2.5)\label{formula1}
Let $f :X \to Y$ be a proper algebraic map of complex algebraic
varieties, with $Y$ irreducible. Let $\VV$ be the set of components
of strata of $Y$ in an algebraic stratification of $f$, and assume
$\pi_1(V)=0$ for all $V \in \VV$. For each $V \in \VV$ with ${\rm
dim} (V)< {\rm dim} (Y)$, define inductively
$$\widehat{I\chi}_y(\bar V)= I\chi_y(\bar V)- \sum_{W < V}
\widehat{I\chi}_y(\bar W) \cdot I\chi_y(c^{\circ}L_{W,V}),$$ where
$c^{\circ}L_{W,V}$ denotes the open cone on the link of  $W$ in
$\bar{V}$. Then: \begin{equation}\label{E20} \chi_y(X)=I\chi_y(Y)
\cdot \chi_y(F) + \sum_{V < S} \widehat{I\chi}_y(\bar V) \cdot
\left( \chi_y(F_V) - \chi_y(F) \cdot I\chi_y(c^{\circ}L_{V,Y})
\right),\end{equation} where $F$ is the (generic) fiber over the
top-dimensional stratum $S$, and $F_V$ is the fiber of $f$ above the
stratum $V \in \VV \setminus \{S\}$.
\end{theorem}

\begin{remark}\rm Formula (\ref{E20}) yields calculations of classical topological
and algebraic invariants of the complex algebraic variety $X$, e.g.
Euler characteristic, and if $X$ is smooth and projective, signature
and arithmetic genus, in terms of singularities of proper algebraic
maps defined on $X$. In particular, if in Theorem \ref{formula1} we take
$f=id$, then formula (\ref{E20}) yields an interesting relationship
between the $\chi_y$- and respectively the $I\chi_y$-genus of an
irreducible complex algebraic variety $Y$:
\begin{equation} \chi_y(Y)=I\chi_y(Y) + \sum_{V < S} \widehat{I\chi}_y(\bar V) \cdot
\left( 1 -  I\chi_y(c^{\circ}L_{V,Y})
\right).\end{equation}

\end{remark}

Similarly, for $X$ pure dimensional, if we let $M=IC'^H_X$ then, in
the above notations and assumptions on the monodromy along the
strata, Proposition \ref{gsmp} yields the following formula (cf.
\cite{CMS} for complete details):
\begin{multline}\label{E21} I\chi_y(X)=I\chi_y(Y) \cdot I\chi_y(F) \\ + \sum_{V < S}
\widehat{I\chi}_y(\bar V) \cdot \left( I\chi_y(f^{-1}(c^\circ
L_{V,Y})) - I\chi_y(F) \cdot I\chi_y(c^{\circ}L_{V,Y})
\right).\end{multline}

By applying the transformation $MHT_y$ to the identity of Corollary
\ref{cormain} for $M=\Q^H_X$, and resp. for $M=IC'^H_X$, and by
using the fact that $MHT_y$ commutes with the exterior product, we
obtain the following result:
\begin{theorem}(\cite{CMS}, Thm. 4.7)\label{charfor}
Let $f :X \to Y$ be a proper  morphism of complex algebraic
varieties, with $Y$ irreducible. Let $\VV$ be the set of components
of strata of $Y$ in a stratification of $f$, with $S$ the
top-dimensional stratum (which is Zariski-open and dense in $Y$),
and assume $\pi_1(V)=0$ for all $V \in \VV$. For each $V \in \VV
\setminus \{S\}$, define inductively
$$\widehat{IT}_{y *}(\bar V):= {IT_y}_*(\bar V)- \sum_{W < V}
\widehat{IT}_{y *}(\bar W) \cdot I\chi_y(c^{\circ}L_{W,V}),$$ where
$c^{\circ}L_{W,V}$ denotes the open cone on the link of  $W$ in
$\bar{V}$, and all homology characteristic classes are regarded in
the Borel-Moore homology of the ambient variety $Y$ (with
coefficients in $\Q[y,y^{-1},(1+y)^{-1}]$). Then:
\begin{equation}\label{charformula}
f_*{T_y}_*(X)={IT_y}_*(Y) \cdot \chi_y(F) + \sum_{V < S}
\widehat{IT}_{y *}(\bar V) \cdot \left( \chi_y(F_V) - \chi_y(F)
\cdot I\chi_y(c^{\circ}L_{V,Y}) \right),
\end{equation}
where $F$ is the generic fiber of $f$, and $F_V$ denotes the fiber
over a stratum $V \in \VV \setminus \{S\}$.

If, moreover, $X$ is pure-dimensional, then:
\begin{multline}\label{charformula'}
f_*{IT_y}_*(X) = {IT_y}_*(Y) \cdot I\chi_y(F)\\
+ \sum_{V < S}
\widehat{IT}_{y *} (\bar V) \cdot \left( I\chi_y(f^{-1}(c^\circ
L_{V,Y})) - I\chi_y(F) \cdot I\chi_y(c^{\circ}L_{V,Y}) \right).
\end{multline}
\end{theorem}

These formulae can be viewed as, on the one hand, yielding powerful
methods of inductively calculating (even parametrized families of)
characteristic classes of algebraic varieties (e.g., by applying
them to resolutions of singularities). On the other hand, they can
be viewed as yielding topological and analytic constraints
on the singularities of any proper algebraic morphism (e.g., even
between smooth varieties), expressed in terms of (even parametrized
families of) their characteristic classes.

\begin{remark}\label{-1}\rm For the value $y=-1$ of the parameter, i.e., in
the case of (intersection (co)homology) Euler characteristics and
MacPherson-Chern homology
characteristic classes, all formulae in this section hold (even in
the compact complex analytic case) without any assumption on the
monodromy along the strata. This fact is a consequence of a formula
similar to (\ref{mE}), which holds in the abelian group of
$\VV$-constructible functions on $Y$ (see \cite{CMS0}, Theorem
3.1(2)).
\end{remark}

It is interesting to see
how the results of this section simplify in the following situation:
\begin{prop} If $f:X \to Y$ is a proper algebraic map between
irreducible $n$-dimensional complex algebraic varieties so that $f$
is \emph{homologically small of degree $1$} (in the sense of
\cite{GM2}, \S 6.2), then
\begin{equation} f_*{IT_y}_*(X)={IT_y}_*(Y) \ \ \text{and} \ \ I\chi_y(X)=I\chi_y(Y).\end{equation}
In particular, if $f:X \to Y$ is a \emph{small resolution}, that is
a resolution of singularities that is also small (in the sense of
\cite{GM2}), then \footnote{Finding numerical invariants of complex
varieties, more precisely Chern numbers that are invariant under
small resolutions, was Totaro's guiding principle in his paper
\cite{To}.}:
\begin{equation}\label{small} f_*{T_y}_*(X)={IT_y}_*(Y)  \ \ \text{and} \ \ \chi_y(X)=I\chi_y(Y).\end{equation}\end{prop}
\begin{proof} Indeed, for such a map we have that $f_*IC_X \in Perv(\Q_Y)$, more precisely there is a
(canonical) isomorphism (\cite{GM2}, Theorem 6.2):
\begin{equation}\label{small'} f_*IC_X \simeq {^{p}\HC^0} (f_*IC_X) \simeq IC_Y \in
D^b_c(Y).\end{equation} Moreover, as $rat: \mh(Y) \to Perv(\Q_Y)$ is
a faithful functor, this isomorphism can be lifted to the level of
mixed Hodge modules. Then, since $MHT_y$ commutes with proper
push-downs and $[IC'^H_X]=(-1)^n[IC^H_X]$ in $K_0(\mh(X))$, we
obtain:
\begin{eqnarray*}f_*{IT_y}_*(X) &=& f_*MHT_y([IC'^H_X])=(-1)^nMHT_y(f_*[IC^H_X])\\
&=& (-1)^nMHT_y([IC^H_Y]) =MHT_y([IC'^H_Y])={IT_y}_*(Y) \:.
\end{eqnarray*}
The claim about genera follows by noting that the isomorphism
(\ref{small'}) (when regarded at the level of mixed Hodge modules)
induces a (canonical) isomorphism of mixed Hodge structures $IH^*(X)
\simeq IH^*(Y)$.

\end{proof}

\subsection{Lifts of characteristic classes to intersection
homology.} For a singular space $Y$, the usual characteristic class
theories are natural transformations taking values in the
(Borel-Moore) homology. If $Y$ is a closed manifold, then by
Poincar\'e Duality these homology characteristic classes are in the
image of the cap product map $$H^{dim_{\R}(Y)-*}(Y) \overset{\cap
[Y]}{\to} H_{*}(Y),$$ so they lift to classes in cohomology. But the
Poincar\'e Duality ceases to hold if the space $Y$ has
singularities. However, if $Y$ is a topological pseudomanifold which
for simplicity we assume to be compact, and for $\bar p$ a fixed
perversity, the cap product map factors through the perversity $\bar
p$ intersection homology groups:
$$H^{dim_{\R}(Y)-*}(Y) \to IH^{\bar p}_*(Y) \to H_{*}(Y).$$
It is therefore natural to ask what homology characteristic classes
of $Y$ admit lifts to intersection homology.
In the case of the topological $L_*$-classes this is not obvious, and discussed
in \cite[(6.2)]{CS2} based on their mapping theorem for these $L_*$-classes.\\

But for a complex algebraic variety $Z$, the MacPherson-Chern class
transformation $c_*$ and the Baum-Fulton-MacPherson Todd class
transformation $td_*$ factorize through the (rationalized) Chow
group $CH_*(Z)_{\Q}$ of $Z$ (cf. \cite{Ke,F}). So the same applies
to the Hirzebruch class transformation $MHT_{y}$ (specialized at any
value of $y$, compare \cite{BSY,SY}). And by a deep result from
\cite{BB,W} (compare also with the more recent \cite{HS}), the image
of the fundamental class map:
$$cl: CH_i(Z)_{\Q}\to H_{2i}(Z;\Q)$$
can be lifted (in general non-uniquely) to the middle intersection homology, i.e.,
$$im(cl: CH_i(Z)_{\Q}\to H_{2i}(Z;\Q))\subset
im(IH^{\bar m}_{2i}(Z;\Q) \to H_{2i}(Z;\Q)) \:.$$

As a corollary, we obtain the following result
\begin{theorem} Let $Z$ be a complete complex algebraic
variety. Then for any rational value $y=a \in \Q$ of the parameter
$y$ the $i$-th piece of the Hirzebruch homology class ${T_a}_*(Z)$,
and for $Z$ pure-dimensional also of the homology class
${IT_a}_*(Z)$, is in the image of the natural map
$$IH^{\bar m}_{2i}(Z;\Q) \to H_{2i}(Z;\Q).$$
\end{theorem}

\begin{remark}\rm The conjectured equality  ${IT_{-1}}_*(Z)=L_*(Z)$ would imply
that the $L$-class $L_*(Z)$ of the pure-dimensional compact complex
algebraic variety $Z$ has a canonical lift to (rationalized) Chow
groups $CH_*(Z)_{\Q}$, and therefore also (non-canonically) to
middle intersection homology $IH^{\bar m}_{2*}(Z;\Q)$.
\end{remark}

\section{The contribution of monodromy. Atiyah-Meyer type
formulae.}\label{monAM}

If we drop the assumption of trivial monodromy along the strata in a
stratification of a proper algebraic morphism, then the right hand
side of the formulae in the previous section should be written in
terms of twisted intersection homology genera and respectively
twisted Hirzebruch characteristic classes. Indeed, for any complex
algebraic variety $Z$ we have the identification
\begin{equation}K_0(\mh(Z))=K_0(\po(Z)^p),\end{equation}
where $\po(Z)^p$ denotes the abelian category of pure polarizable
Hodge modules. And by the decomposition by strict support, it
follows that $K_0(\po(Z)^p)$ is generated by elements of the form
$[IC^H_S(\LL)]$, for $S$ an irreducible closed subvariety of $Z$ and
$\LL$ a polarizable variation of Hodge structures (admissible at
infinity) defined on a smooth Zariski open and dense subset of $S$.
Thus the image of the natural transformation $MHT_y$ is generated by
twisted characteristic classes ${IT_y}_*(S;\LL)$, for $S$ and $\LL$
as before. It is therefore natural to look for Atiyah-Meyer type
formulae for the twisted Hirzebruch classes.

\bigskip

The central result of this section is the following Meyer-type
formula for twisted Hirzebruch classes of algebraic manifolds (see
\cite{CLMS} for complete details), whose proof is included here for
the sake of completeness:
\begin{theorem}(\cite{CLMS})\label{Meyerc}
Let $Z$ be a complex algebraic manifold of pure dimension $n$, and
$\LL$ an admissible variation of mixed Hodge structures on $Z$ with
associated flat bundle with Hodge filtration $(\VV,
\mathcal{F}^{\bullet})$. Then
\begin{equation}\label{WMc}
{T_y}_*(Z;\LL)= \left( ch^*_{(1+y)}(\chi_y(\VV)) \cup T_y^*(T_Z)
\right) \cap [Z]= ch^*_{(1+y)}(\chi_y(\VV)) \cap {T_y}_*(Z),
\end{equation}
where
$$\chi_y(\VV):=\sum_p \left[Gr^{p}_{\mathcal{F}} \VV \right]
\cdot (-y)^{p} \in K^0(Z)[y,y^{-1}]$$
is the $K$-theory $\chi_y$-characteristic of $\VV$ (with $K^0(Z)$ the Grothendieck
group of algebraic vector bundles on $Z$),
and $ch^*_{(1+y)}$ is the twisted Chern character defined in Section
\ref{Hir}.
\end{theorem}
\begin{proof}
Let $\mathcal{V}:=\LL \otimes_{\Q} \mathcal{O}_Z$ be the flat bundle
with holomorphic connection $\bigtriangledown$, whose sheaf of
horizontal sections is $\LL \otimes \C$. The bundle $\VV$ comes
equipped with its Hodge (decreasing) filtration by holomorphic
sub-bundles $\FC^p$, and these are required to satisfy the
Griffiths' transversality condition $$\bigtriangledown(\FC^p)
\subset \Omega^1_Z \otimes \FC^{p-1}.$$ The bundle $\VV$ becomes a
holonomic $D$-module bifiltered by
$$W_k\VV:=W_k\LL \otimes_{\Q} \mathcal{O}_Z,$$
$$F_p\VV:=\FC^{-p}\VV.$$
This data constitutes the smooth mixed Hodge module $\LL^H[n]$. It
follows from Saito's work that there is a filtered quasi-isomorphism
between $(DR(\LL^H),F_{-{\bullet}})$ and the usual filtered de Rham
complex $(\Omega_Z^{\bullet}(\mathcal{V}), F^{\bullet})$ with the
filtration induced by Griffiths' transversality, that is,
$$F^p \Omega_Z^{\bullet}(\mathcal{V}):= \left[ \mathcal{F}^p
\overset{\bigtriangledown}{\to} \Omega_Z^1 \otimes \mathcal{F}^{p-1}
\overset{\bigtriangledown}{\to} \cdots
\overset{\bigtriangledown}{\to} \Omega_Z^i \otimes \mathcal{F}^{p-i}
\overset{\bigtriangledown}{\to} \cdots \right].$$ Therefore,
{\allowdisplaybreaks
\begin{eqnarray*}
MHC_*([\LL^H]) &=& \sum_{p,i} (-1)^{i} [\HC^i ( gr^F_{-p}
DR(\LL^H) )]  \cdot (-y)^p \\ &=& \sum_{p,i} (-1)^{i} [\HC^i (
gr_F^{p} \Omega_Z^{\bullet}(\mathcal{V}) )]  \cdot (-y)^p \\ &=&
\sum_{p,i} (-1)^{i} [\Omega_Z^{i} \otimes
Gr_{\mathcal{F}}^{p-i}\mathcal{V}] \cdot (-y)^p \\
&=& \chi_y(\VV) \otimes \lambda_y(T^*_Z) \in G_0(Z) \otimes
\Z[y,y^{-1}],
\end{eqnarray*} }
where $\lambda_y(T^*_Z):=\sum_p \Lambda^pT^*_Z\cdot y^p$ the total
$\lambda$-class of $Z$. Since $Z$ is an algebraic manifold, the Todd
class transformation of the classical Grothendieck-Riemann-Roch
theorem is explicitly described by \footnote{This formula is the
counterpart of the Atiyah-Meyer formula in the coherent context of
the Todd-class transformation of Baum-Fulton-MacPherson
(\cite{BFM}). More generally, the counterpart of the
Banagl-Cappell-Shaneson formula (\ref{BCS}) in the coherent context
is $td_*(\mathcal{G})=ch^*([\mathcal{G}]) \cap td_*(Z)$, for a
locally free coherent sheaf $\mathcal{G}$ on the singular algebraic
variety $Z$.}
$$td_*(\cdot)=ch^*(\cdot)td^*(Z) \cap [Z].$$ Therefore, by applying
$td_*$ (which is linearly extended over $\Z[y,y^{-1}]$) to the above
equation, we have that
\begin{equation}\label{i4}
td_* \left( MHC_*([\LL^H]) \right)=\left(
ch^*(\chi_y(\mathcal{V})) \cup \tilde{T}_y^*(T_Z) \right) \cap [Z],
\end{equation}
where $\tilde{T}_y^*(T_Z):=ch^*(\lambda_y(T^*_Z)) \cup td^*(Z)$ is
the un-normalized Hirzebruch class (in cohomology). The claimed
formula (\ref{WMc}) follows now from the definition of $td_{(1+y)}$,
by noting that the identities
$$ch^*_{(1+y)}(\cdot)_{2k}=(1+y)^k
\cdot ch^*(\cdot)_{2k}, \quad \text{and} \quad  T_y^k(T_Z)=(1+y)^{k-n} \cdot
\tilde{T}^k_y(T_Z)$$
hold in $H^{2k}(Z) \otimes \Q[y]$. Indeed, we
have in $H_{2k}^{BM}(Z) \otimes \Q[y,y^{-1}]$ the following sequence
of equalities
\begin{eqnarray*}
&& td_k \left( MHC_*([\LL^H]) \right) = \left(
ch^*(\chi_y(\mathcal{V})) \cup \tilde{T}_y^*(T_Z) \right)^{2(n-k)}
\cap [Z]\\ && \qquad= \left( \sum_{i+j=n-k} ch^*(\chi_y(\mathcal{V}))_{2i}
\cup \tilde{T}_y^j(T_Z) \right) \cap [Z] \\ && \qquad= \left(
\sum_{i+j=n-k} (1+y)^{-i} ch^*_{(1+y)}(\chi_y(\mathcal{V}))_{2i}
\cup (1+y)^{n-j} T_y^j(T_Z) \right) \cap [Z]\\ && \qquad= (1+y)^k \left(
ch^*_{(1+y)}(\chi_y(\mathcal{V})) \cup T_y^*(T_Z) \right)^{2(n-k)}
\cap [Z].
\end{eqnarray*}

\end{proof}

\begin{cor}
If the variety $Z$ in Theorem \ref{Meyerc} is also \emph{complete},
then by pushing down to a point, we obtain a Hodge theoretic
Meyer-type formula for the twisted $\chi_y$-genus:
\begin{equation}\label{gWM}
\chi_y(Z;\LL)= \langle ch^*_{(1+y)}(\chi_y(\VV)) \cup T_y^*(T_Z),
[Z] \rangle.
\end{equation}
\end{cor}

\begin{remark}\rm Assume that the local system $\LL$ underlies a
polarizable variation of pure Hodge structures of weight $i$ on $Z$.
Then the choice of such a polarization defines after identifying the
Tate twists $\Q_Z(i)\simeq \Q_Z$ a suitable duality structure on $\LL$, i.e.
makes it a Poincar\'e local system. Then it is easy to see that the
image of $\chi_1(\VV)$ under the natural map
$$can: K^0(Z)\to KU(Z) \to KU(Z)[1/2] \supset KO(Z)[1/2]$$
agrees with the $K$-theory signature $[\LL]_K$ of this Poincar\'e local system.
So this class
$$can(\chi_1(\VV))\in KU(Z)[1/2]$$
does not depend on the choice of the polarization. In the same way
one also gets for $Z$ projective the equality
$$ \chi_1(Z;\LL)=\sigma(Z;\LL)\:,$$
so that in this case the formula (\ref{gWM}) exactly specializes for $y=1$ to
Meyer's signature formula (\ref{M}). Recall that $T_1^*(T_Z)=L^*(T_Z)$
for $Z$ smooth.

Similarly, for any variation of mixed Hodge structures one gets by
definition that
$$\chi_{-1}(\VV)=[\VV]\in K^0(Z)\quad \text{and} \quad
ch^*_{(0)}(\chi_{-1}(\VV))={\mbox{rk}}(\VV)={\mbox{rk}}(\LL)\in H^0(Z;\Q) \:.$$ So the
formula (\ref{gWM}) specializes for $y=-1$ to the well-known formula
for the  Euler characteristic of $Z$ with coefficients in $\LL$:
$$\chi(H^*(Z;\LL))= {\mbox{rk}}(\LL)\cdot \chi(H^*(Z;\Q))=
{\mbox{rk}}(\LL)\cdot \langle c^*(T_Z),[Z] \rangle \:.$$
\end{remark}

\begin{remark}\label{nonc}\rm Without the compactness assumption on $Z$, we can obtain
directly a formula for $\chi_y(Z;\LL)$ by noting that the twisted
logarithmic de Rham complex $\Omega_X^{\bullet}(\log D) \otimes
\bar{\VV}$ associated to the Deligne extension of $\LL$ on a good
compactification $(X,D)$ of $Z$ (with $X$ smooth and compact, and
$D$ a simple normal crossing divisor), with its Hodge filtration
induced by Griffiths' transversality, is part of a cohomological
mixed Hodge complex that calculates $H^*(Z; \LL \otimes \C)$. In the
above notation, we then obtain (cf. \cite{CLMS}, Theorem 4.10):
\begin{equation}\label{nc}
\chi_y(Z;\LL)=\langle ch^* (\chi_y(\bar{\VV})) \cup ch^* \left(
\lambda_y(\Omega_X^{1}({log}D)) \right) \cup td^*(X), [X] \rangle.
\end{equation}
Here $\langle , \rangle$ denotes the Kronecker pairing on $X$,
$td^*(X):=td^*(T_X)$ is the total Todd class of $X$ (in cohomology),
$$\lambda_y \left( \Omega_X^{1}({\log}D) \right):=\sum_i
\Omega_X^{i}({\log}D) \cdot y^i, \quad \text{and} \quad
\chi_y(\bar{\VV})=\sum_p \left[ Gr^{p}_{\bar{\mathcal{F}}}
\bar{\mathcal{\VV}} \right] \cdot (-y)^{p}\:,$$ with $(\bar{\VV},
\bar{\mathcal{F}^{\bullet}})$ the unique extension of
$(\VV,\mathcal{F}^{\bullet})$ to $X$ corresponding to the Deligne
extension of $\LL$ (cf. \cite{De}).

For future reference, we mention here a different way of proving
formula (\ref{nc}). Under the above notations and for $j:Z
\hookrightarrow X$ the inclusion map, Saito's work implies that
there is a filtered quasi-isomorphism between
$(DR(j_*\LL^H),F_{-{\bullet}})$ and the usual filtered logarithmic
de Rham complex of $(\bar{\VV}, \bar{\mathcal{F}^{\bullet}})$. Then,
as in the proof of Theorem \ref{Meyerc}, it follows that
\begin{equation}\label{nc2}
MHC_*([j_*\LL^H])=\chi_y(\bar{\VV}) \otimes \lambda_y \left(
\Omega_X^{1}(\log D) \right)
\end{equation}
(Note that all coherent sheaves appearing in the above formula are
locally free). Therefore, by applying the transformation $td_*$
(which is linearly extended over $\Z[y,y^{-1}]$) to the above
equation, we have that
\begin{equation}\label{nc3}
td_* \left( MHC_*([j_*\LL^H]) \right)= \left(
ch^*(\chi_y(\bar{\VV})) \cup ch^* \left( \lambda_y(\Omega_X^{1}(\log
D)) \right) \cup td^*(X) \right) \cap [X].
\end{equation}
Formula (\ref{nc}) can be now obtained by pushing  (\ref{nc3}) down
to a point via the constant map $k:X \to pt$, and by using an
argument similar to that of [\cite{CLMS}, Proposition 5.4].
\end{remark}

In the relative setting, as an application of Theorem \ref{Meyerc}
we obtain the following Atiyah-type result:
\begin{theorem}(\cite{CLMS})\label{gAMc} Let $f:E \to B$ be a projective morphism of complex
algebraic varieties, with $B$ smooth and connected. Assume that the
sheaves $R^sf_*\Q_E$, $s \in \Z$, are locally constant on $B$, e.g.,
$f$ is a locally trivial topological fibration. Then
\begin{equation}\label{gtwistedc}
f_*{T_y}_*(E)= ch^*_{(1+y)}\left( \chi_y (f) \right) \cap
{T_y}_*(B),
\end{equation}
where
$$\chi_y(f):=\sum_{i,p} (-1)^i \left[ Gr^p_{\mathcal{F}} \HC_i
\right] \cdot (-y)^p \in K^0(B)[y]$$
is the $K$-theory
$\chi_y$-characteristic of $f$, for $\HC_i$ the flat bundle with
connection $\bigtriangledown_i: \HC_i \to \HC_i
\otimes_{\mathcal{O}_B} \Omega^1_B$, whose sheaf of horizontal
sections is  $R^if_*\C_E$.

If, moreover, $B$ is complete, then by pushing down to a point, we
obtain:
\begin{equation}\label{gtw}
\chi_y(E)= \langle ch^*_{(1+y)}\left( \chi_y (f) \right) \cup
T_y^*(T_B), [B] \rangle.
\end{equation}
\end{theorem}

\begin{proof}
If in (\ref{i1}) we let $M^{\bullet}=f_*\Q^H_E$, then by using
(\ref{i2}) we obtain the following identity in $K_0(\mh(B))$:
\begin{equation}\label{i5} \left[ f_*\Q^H_E \right]=\sum_{i \in \Z} (-1)^i [H^i(f_* \Q^H_E)]=\sum_{i \in \Z}
(-1)^i \left[ H^{i+\text{dim}B}(f_* \Q^H_E)[-\text{dim}B]\right].
\end{equation}
Note that $H^{i+\text{dim}B}(f_* \Q^H_E) \in \mh(B)$ is the smooth
mixed Hodge module on $B$ whose underlying rational complex is
(recall that $B$ is smooth)
\begin{equation} rat (H^{i+\text{dim}B}(f_*
\Q^H_E))={^p\HC}^{i+\text{dim}B}(Rf_*\Q_E)=(R^i
f_*\Q_E)[\text{dim} B],
\end{equation}
where ${^p\HC}$ denotes the perverse cohomology functor. In this
case, each of the local systems $\LL_i:=R^{i} f_*\Q_E$ underlies a
geometric (hence admissible) variation of Hodge structures. By
applying the natural transformation $MHT_y$ to the equation
(\ref{i5}), and using the fact that $f$ is proper, we have that
$$f_*{T_y}_*(E)=\sum_i(-1)^i {T_y}_*(B;\LL_i).$$
In view of Theorem \ref{Meyerc} this yields the formula in equation
(\ref{gtwistedc}).

\end{proof}

\begin{remark}\rm If the monodromy action of $\pi_1(B)$ on $H^*(F)$ is
\emph{trivial} (e.g., $\pi_1(B)=0$), i.e., if the local systems
$R^{i} f_*\Q_E$ ($i \in \Z$) are constant on $B$, then by the
``rigidity theorem" (e.g., see the discussion in the last paragraph
of \cite{CMS}, \S 3.1) the underlying variations of mixed Hodge
structures are constant, so that
\begin{equation}ch^*_{(1+y)}\left( \chi_y (f)
\right)=\chi_y(F) \in H^0(B;\Q[y,y^{-1}]).\end{equation} In this case, formula
(\ref{gtw}) yields the multiplicative relation
$$\chi_y(E)=\chi_y(F) \cdot \chi_y(B)\:,$$
thus extending the Chern-Hirzebruch-Serre theorem
(in the context of complex algebraic varieties).
\end{remark}

\bigskip

Theorem \ref{Meyerc} can also be used for computing invariants
arising from intersection homology (cf. Definition \ref{IHinv}). In
the above notations, we have the following
\begin{prop}(\cite{CLMS2})\label{IH} Let $f:E \to B$ be a proper
morphism of complex algebraic varieties, with $E$ pure-dimensional
and $B$ smooth and connected. Assume that $f$ is a locally trivial
topological fibration with fiber $F$. Then
\begin{equation}\label{str}
f_*{IT_y}_*(E)=\sum_i (-1)^{dim F +i} \ {T_y}_*(B;\LL_i),
\end{equation}
where $\LL_i$ is the admissible variation of mixed Hodge structures on
$B$ with stalk $IH^{dim F  +i}(F;\Q)$ and with associated smooth mixed
Hodge module $H^i(f_*IC^H_E) \in \mh(B)$.
\end{prop}

\begin{proof}
The following equation in $K_0(\mh(B))$ is a consequence of the
identities (\ref{i1}) and (\ref{i2}):
\begin{equation}\label{G} [f_*IC_E^H]=\sum_i(-1)^i \left[ H^i(f_*IC_E^H)
\right]=(-1)^{{\rm dim}(B)} \cdot \sum_i (-1)^i\left[
H^i(f_*IC_E^H)[-{\rm dim}(B)]\right].\end{equation} Note that
$H^i(f_* IC^H_E) \in \mh(B)$ is the smooth mixed Hodge module on $B$
whose underlying rational complex is
\begin{equation} rat (H^i(f_* IC^H_E))={^p\HC}^i(Rf_*IC_E)=(R^{i-\text{dim}B}
f_*IC_E)[\text{dim} B],
\end{equation}
where the second equality above follows since $B$ is smooth (hence smooth
perverse sheaves are, up to a shift, just local systems on $B$). In
particular, each of the local systems $\LL_i:=R^{i-\text{dim}B}
f_*IC_E$ ($i \in \Z$) underlies an admissible variation of mixed
Hodge structures.

By applying the natural transformation $MHT_y$ to the equation
(\ref{G}), and using the fact that $MHT_y$ commutes with $f_*$
(since $f$ is proper), we obtain the formula in equation
(\ref{str}).

It remains to identify the stalks of the local systems $\LL_i$ ($i
\in \Z$). Let $b \in B$ with $i_b:\{b\} \hookrightarrow B$ the
inclusion map. Then $\{f=b\}$ is the (general) fiber $F$ of $f$, so
it is locally normally nonsingular embedded in $E$. It follows that
we have a quasi-isomorphism $IC_E|_F \simeq IC_F[{\rm codim}F]$
(e.g., see \cite{GM2}, \S 5.4.1). Then by proper base change we
obtain that
\begin{eqnarray*}
(\LL_i)_b &=& (R^{i-\text{dim}B}f_*IC_E)_b=\HC^{i-\text{dim}B}(i_b^*Rf_*IC_E)\\
&=& IH^{i-\text{dim}B+\text{dim}E}(F;\Q)=
IH^{i+\text{dim}F}(F;\Q) \:.
\end{eqnarray*}

\end{proof}

Each term in the right hand side of equation (\ref{str}) can be
computed by formula (\ref{WMc}). Let $\VV_i$ be the flat bundle with
connection associated to the admissible variation of mixed Hodge
structures $\LL_i:=R^{i-\text{dim}B} f_*IC_E$, that is $\VV_i:=\LL_i
\otimes_{\Q} \mathcal{O}_B$. Recall that this comes equipped with a
filtration by holomorphic sub-bundles satisfying Griffiths'
transversality. Define the $I\chi_y$-characteristic of $f$ by
\begin{equation}I\chi_y(f):=\sum_i (-1)^{i+\text{dim}F}\cdot
\chi_y(\VV_i).\end{equation} Then as a consequence of (\ref{WMc}),
the above proposition yields the following
\begin{cor} Under the notations and assumptions of Propositions
\ref{IH}, we obtain
\begin{equation}\label{IHm}
f_*{IT_y}_*(E)=ch^*_{(1+y)}(I\chi_y(f)) \cap {T_y}_*(B).
\end{equation}
In particular, if $\pi_1(B)=0$, then $f_*{IT_y}_*(E)=I\chi_y(F)
\cdot {T_y}_*(B)$. \end{cor}

The last assertion of the corollary follows since, under the trivial
monodromy assumption, we have that
$$ch^*_{(1+y)}(I\chi_y(f))=I\chi_y(F) \in H^0(B;\Q[y,y^{-1}])\:.$$
Similar
considerations apply to genera. This is a very special case of the
stratified multiplicative property studied in detail in \cite{CMS}
and summarized in Section \S \ref{SMP} above.


\subsection{Atiyah-Meyer formulae in intersection homology}
We conclude this report with a result from work in progress
(\cite{MS}) on the computation of twisted intersection homology
genera. The following theorem can be regarded as a Hodge-theoretic
analogue of the Banagl-Cappell-Shaneson formula (\cite{BCS}):

\begin{theorem}(\cite{MS})
Assume $i: Z \hookrightarrow M$ is the closed inclusion of an
irreducible (or pure-dimensional) algebraic subvariety into the
smooth algebraic manifold $M$, with $\LL$ a local system on $M$
underlying an admissible variation of mixed Hodge structures with
associated flat bundle $(\VV, \mathcal{F}^{\bullet})$. Then one has
the formula:
\begin{equation}{IT_y}_*(Z;i^*\LL)= i^*(ch^*_{(1+y)}(\chi_y(\VV)))\cap
{IT_y}_*(Z) = ch^*_{(1+y)}(i^*(\chi_y(\VV)))\cap
{IT_y}_*(Z).\end{equation}
\end{theorem}

\begin{proof}
One has for the underlying perverse sheaves the equality:
$$IC_Z(\LL)= IC_Z \otimes i^*\LL \quad \text{and, after shifting,}
\quad IC'_Z(\LL)= IC'_Z \otimes i^*\LL \:.$$
And this implies on the level of (shifted)
mixed Hodge modules that:
$$IC^H_Z(\LL)= IC^H_Z \otimes i^*\LL^H \quad \text{and
resp.} \quad IC'^H_Z(\LL)= IC'^H_Z \otimes i^*\LL^H \:.$$
So the stated
formula is a special case of the following more general result
for any $M \in K_0(\mh(Z))$:
\begin{equation}
MHT_y([M \otimes i^*\LL^H]) = ch^*_{(1+y)}(i^*(\chi_y(\VV))) \cap
MHT_y([M])\:. \end{equation}

By resolution of singularities, we see that the Grothendieck group
$K_0(\mh(Z))$ is generated by elements of the form
$[p_*(j_*\LL'^H)]$ with $p: X \to Z$ a proper algebraic map from a
smooth  algebraic manifold $X$, $j: U=X \setminus D
\hookrightarrow X$ the open inclusion of the complement of a normal
crossing divisor $D$ with smooth irreducible components, and $\LL'$
an admissible variation of mixed Hodge structures on $U$. But
$MHT_y$ commutes with proper pushdown, and $ch^*_{(1+y)}$ commutes
with pullbacks, so that by the projection formula it is enough to
show that:
\begin{equation}\label{sts} MHT_y([j_*\LL'^H \otimes p^*i^*\LL^H])=
ch^*_{(1+y)}(p^*i^*(\chi_y(\VV))) \cap
MHT_y([j_*\LL'^H]).\end{equation}

At this point we can use the identity of formula (\ref{nc3}),
already discussed in Remark \ref{nonc}:
\begin{equation}\label{ext}
td_*\left(MHC_*([j_*\LL'^H])\right) = ch^*(\chi_y(\bar{\VV}')) \cup
ch^*(\lambda_y(\Omega^1_{X}(\log D))) \cap td_*(X) \:,
\end{equation}
with $td_*(X)=td^*(T_X)\cap [X]$, and $\bar{\VV}'$ the Hodge bundle of
the Deligne extension of $\LL'$ to $(X,D)$. Moreover, we also have
that:
$$Rj_*(\LL') \otimes p^*i^*\LL = Rj_*(\LL'\otimes j^*p^*i^*\LL)$$
and similarly on the level of (shifted) mixed Hodge modules, so that
$$ch^*(\chi_y(\bar{\VV}' \otimes j^*p^*i^*\VV))=
ch^*(\chi_y(\bar{\VV}')) \cup ch^*(\chi_y(p^*i^*\VV)).$$ From here
the stated formula follows (as in Theorem \ref{Meyerc}) by the usual
recalculation in terms of $ch^*_{(1+y)}$.
\end{proof}

\begin{remark}\rm The formula (\ref{ext}) can be also be used for
showing the following important facts (cf. \cite{MS}):
\begin{enumerate}
\item The motivic Hirzebruch transformation $MHT_y$ commutes with exterior products.
\item The limit ${IT_y}_*(Z;\LL)$ for $y=-1$ always exists,
as well as ${IT_y}_*(Z;\LL)$ for $y=0$, if $\LL$ is of non-negative weight,
e.g. $\LL=\Q_Z$.
\item More generally the limit ${MHT_y}([M])$ for $y=-1$ always exists
for any mixed Hodge module $M$ on $Z$, with
$${MHT_{-1}}([M])=c_*([rat(M)])\otimes\Q$$
the rationalized MacPherson-Chern class of the underlying perverse sheaf,
i.e. of the corresponding constructible function given by the
Euler characteristics of the stalks.
\end{enumerate} \end{remark}

\bibliographystyle{amsalpha}

\end{document}